\documentclass[12pt]{article}
\usepackage{times}
\usepackage{amsmath,amsfonts,amstext,amssymb,amsbsy,amsopn,amsthm,eucal}
\usepackage{txfonts}
\usepackage{dsfont}
\usepackage{graphicx}   

\numberwithin{equation}{section}
  \theoremstyle{plain}
 \newtheorem{theorem}[equation]{Theorem}

 \newtheorem{lemma}[equation]{Lemma}
 \newtheorem{corollary}[equation]{Corollary}

 \theoremstyle{remark}

 \newtheorem{remark}[equation]{Remark}

\theoremstyle{definition}
 \newtheorem{definition}[equation]{Definition}

\newcommand{\abs}[1]{\left\lvert#1\right\rvert}

\newcommand{\Vol}{{\rm Vol}}

\newcommand{\dR}{\mathds{R}}

\newcommand{\cB}{\mathcal{B}}

\newcommand{\cD}{\mathcal{D}}

\newcommand{\cM}{\mathcal{M}}
\newcommand{\cN}{\mathcal{N}}

\newcommand{\cS}{\mathcal{S}}

\newcommand{\cW}{\mathcal{W}}

\begin{document}

\title{Quantitative Stratification and the\\Regularity of Mean Curvature Flow}

\author{Jeff Cheeger\thanks{The author was partially supported by NSF grant DMS1005552}, Robert Haslhofer and Aaron Naber\thanks{The author was partially supported by NSF postdoctoral grant 0903137}}
\date{}
\maketitle

\begin{abstract}
Let $\cM$ be a Brakke flow of $n$-dimensional surfaces in $\mathbb{R}^N$. The singular set $\cS\subset\cM$ has a stratification $\cS^0\subset\cS^1\subset\ldots\cS$, where $X\in \cS^j$ if no tangent flow at $X$ has more than $j$ symmetries. Here, we define quantitative singular strata $\cS^j_{\eta,r}$ satisfying $\cup_{\eta>0}\cap_{0<r} \cS^j_{\eta,r}=\cS^j$. Sharpening the known parabolic Hausdorff dimension bound $\dim \cS^j\leq j$, we prove the effective Minkowski estimates that the volume of $r$-tubular neighborhoods of $\cS^j_{\eta,r}$ satisfies $\Vol (T_r(\cS^j_{\eta,r})\cap B_1)\leq Cr^{N+2-j-\varepsilon}$. Our primary application of this is to higher regularity of Brakke flows starting at $k$-convex smooth compact embedded hypersurfaces. To this end, we prove that for the flow of $k$-convex hypersurfaces, any backwards selfsimilar limit flow with at least $k$ symmetries is in fact a static multiplicity one plane.  Then, denoting by $\cB_r\subset\cM$ the set of points with regularity scale less than $r$, we prove that $\Vol(T_r(\cB_r))\leq C r^{n+4-k-\varepsilon}$. This gives $L^p$-estimates  for the second fundamental form for any $p<n+1-k$.  In fact, the estimates are much stronger and give $L^p$-estimates for the inverse of the regularity scale.  These estimates are sharp.  The key technique that we develop and apply is a parabolic version of the quantitative stratification method introduced in \cite{CheegerNaber_Ricci,CheegerNaber_HarmonicMinimal}.
\end{abstract}

\small{\tableofcontents}


\section{Introduction and main results}

In this paper, we prove estimates and quantitative regularity results for the mean curvature flow of $n$-dimensional surfaces in $\dR^N$ \cite{Brakke_book,Huisken_convex}. Smooth solutions of the mean curvature flow are given by a smooth family of submanifolds $M^n_t\subset \dR^N$ satisfying the evolution equation,
\begin{equation}\label{eqn_MCF}
\partial_t x=H(x),\qquad x\in M_t \, .
\end{equation}
More generally, $M_t$ is a family of Radon-measures that is integer $n$-rectifiable for almost all times and satisfies (\ref{eqn_MCF}) in the weak sense of Brakke, i.e.
\begin{equation}\label{eqn_weakMCF}
\overline{D}_t\int \varphi dM_t\leq \int\left(-\varphi H^2+\nabla \varphi\cdot H\right) dM_t
\end{equation}
for all nonnegative test functions $\varphi$, where $\overline{D}_t$ is the limsup of difference quotients, see Section \ref{prelim} for details. Brakke flows enjoy wonderful existence and compactness properties, see the fundamental work of Brakke and Ilmanen \cite{Brakke_book,Ilmanen_EllipticRegularization}.  The main problem is then to investigate their regularity.\\

Our results build upon the deep regularity theory of Brian White \cite{White_anno,White_stratification,White_dimension,White_nature,White_subsequent}, see below for a short summary. The main new technique that we develop and apply is, with numerous tweaks and complementary results, a parabolic version of the quantitative stratification method introduced in \cite{CheegerNaber_Ricci,CheegerNaber_HarmonicMinimal} (see also \cite{cheegerqdg} for a general perspective). This method enables us to turn infinitesimal statements arising from blowup arguments and dimension reduction, into much more quantitative estimates.\\

To outline our main results we begin with a brief and imprecise recollection of the stratification of a Brakke flow, see Section \ref{ss:basic_notation} for a rigorous account.  Given a Brakke flow $\cM$, recall that there exists at every point $X\in \cM$ a tangent flow $\cM_X$ obtained by blow up at $X$.  The tangent flows at $X$ need not be unique, but all tangent flows are self-similar.  One may proceed to separate points $X$ of $\cM$ into strata $\cS^j(\cM)$ based on the number of symmetries for the tangent flows at $X$, again see Section \ref{ss:basic_notation} for a rigorous definition.  At least in spirit, for instance if $\cM$ were a stratified space, then this stratification agrees with the natural one based on the singularities of $\cM$.  For general Brakke flows White first proved the (parabolic) Hausdorff dimension estimate
\begin{align}\label{e:dim}
\dim \cS^j(\cM)\leq j\, .
\end{align}

Our first main result, Theorem \ref{t:Brakke_quant_strat}, valid for Brakke flows $\cM=\{(M_t,t)|t\in I\}$ in general codimension and without additional assumptions, is a quantitative refinement of (\ref{e:dim}).  Specifically, we introduce in Section \ref{ss:quant_reg} what we call the quantitative strata $\cS^j_{\eta,r}(\cM)$ of $\cM$ (Definition \ref{d:Brakke_singularstrata}).  The quantitative stratification does not see the infinitesimal behavior of $\cM$, and instead separates point of $\cM$ based on the {\it almost symmetries} of parabolic balls of $\cM$ of definite radius.  Nonetheless, it is possible to recover the standard stratification from the quantitative stratification, see (\ref{relation}).  For the quantitative stratification we will replace (\ref{e:dim}) with much stronger estimates on the volume of tubular neighborhoods of the quantitative singular strata.  Using (\ref{relation}) we can recover the estimates of (\ref{e:dim}) from this.  In fact, an easy consequence of this will be to replace Hausdorff dimension estimates for the singular set with the much stronger Minkowski dimension estimates.  The quantitative stratification also plays a key role in our regularity theory for the mean curvature flow of $k$-convex hypersurfaces, see below for a general description of this and Section \ref{ss:quant_reg} for rigorous statements.\\

The proof of the estimates of Theorem \ref{t:Brakke_quant_strat} require several new techniques, originally introduced in \cite{CheegerNaber_Ricci,CheegerNaber_HarmonicMinimal} and modified appropriately here to the parabolic setting.  We develop a new blow up technique for mean curvature flows, distinct from those of the standard dimension reduction.  The blow up technique allows for a quantitative refinement, which does not seem accessible by the ideas of the standard dimension reduction.  The key ideas are that of an {\it energy decomposition}, {\it cone-splitting}, and {\it quantitative differentiation}.  We begin by proving an quantitative version of the statement that every tangent flow is self-similar.  Essentially, at any given point we have that away from a definite number of scales that every parabolic ball is {\it almost} self-similar.  The proof of this statement is surprisingly short, and based on a quantitative differentiation argument.  If we refer to as {\it bad} those scales at a point $X$ for which the parabolic ball is not almost self-similar, then the energy decomposition of Section \ref{ss:energy_decomposition} is based on grouping those points of a Brakke flow $\cM$ with the same {\it bad} scales.  The estimates of Theorem \ref{t:Brakke_quant_strat} turn out to be more easily provable on each piece of this decomposition, because such points on good scales interact with one another to force higher order symmetries, an idea we call {\it cone-splitting}.  The final step of the proof is then to show that the number of elements to the energy decomposition is much smaller than one would initially have expected, a result equivalent to the bound on the number of {\it bad} scales from above, and so we can sum the estimates over each piece separately to obtain Theorem \ref{t:Brakke_quant_strat}.\\ 

Our primary application of the quantitative stratification is to the higher regularity theory of some mean curvature flows.  For such applications we focus on the case where the Brakke flow starts at a smooth compact embedded hypersurface $M_0^n\subset \dR^{n+1}$ satisfying some convexity assumption. We consider the case that $M_0$ is $k$-convex, i.e. $\lambda_1+\ldots \lambda_k\geq 0$, where $\lambda_1\leq\ldots\leq \lambda_n$ are the principal curvatures, see e.g. \cite{HuiskenSinestrari_Convexity}. Special instances are the $1$-convex case ($A\geq 0$), where the mean curvature flow converges to a round point \cite{Huisken_convex}, the 2-convex case in which Huisken-Sinestrari proved the existence of a mean curvature flow with surgery \cite{HuiskenSinestrari_Surgery}, and the general mean-convex case ($h\geq 0$) for $k=n$ to which White's regularity theory applies.  \\

Our second main result, Theorem \ref{t:Brakke_cylinders}, states that in the $k$-convex case any backwards selfsimilar limit flow with at least $k$ symmetries is in fact a static multiplicity one plane.  This has been understood previously at the first singular time by use of a maximum principle, but we extend this result into the singular region by using ideas of elliptic regularization from \cite{Ilmanen_EllipticRegularization,White_subsequent}.  A principal application of this result is the proof of a new $\varepsilon$-regularity theorem for $k$-convex mean curvature flows, see Theorem \ref{t:Brakke_epsregularity}.  In short, the theorem states that if in some parabolic ball the mean curvature flow contains enough almost symmetries, then the flow is smooth with definite estimates on a slightly smaller ball.  This will be combined with the quantitative stratification in order to prove our regularity theory for $k$-convex flows.\\

Our third main result, also in the $k$-convex case, is Theorem \ref{t:Brakke_regularity}, which is our main regularity theorem.  It arises as a combination of the quantitative stratification of Theorem \ref{t:Brakke_quant_strat} and the $\varepsilon$-regularity which results from Theorem \ref{t:Brakke_cylinders}.  Roughly, the assertion is that away from a set of small volume the Brakke flow can be written as a smooth single valued graph of definite size with definite estimates.  A little more precisely, let us define the regularity scale $r_{\cM}(X)$ as the supremum of $0\leq r\leq 1$ such that $M_{t'}\cap B_r(x)$ is a smooth graph for all $t-r^2< t'< t+r^2$ and such that
\begin{equation}
\sup_{X'\in\cM\cap B_r(X)} r\, \abs{A}\!(X')\leq 1\, ,
\end{equation}
see also Definition \ref{d:regularity_scale}.  Then Theorem \ref{t:Brakke_regularity} obtains sharp $L^p$ estimates for $r_{\cM}^{-1}$.  As a consequence, though in fact this statement is much weaker, we obtain sharp $L^p$-estimates for the second fundamental form and its derivatives, see Corollary \ref{t:Brakke_estimates}.\\

\begin{remark}
Although we concentrate here on the mean curvature flow, the quantitative stratification techniques can be applied also to other nonlinear evolution equations. In particular, similar results hold for the harmonic map flow. We will discuss this in a subsequent paper \cite{CHN_hmf}.
\end{remark}

\subsection{Basic notions and the regularity theory of Brian White}\label{ss:basic_notation}

Before giving the technical statements of our main results on the mean curvature flow, let us collect some basic notions and some fundamental regularity results from the work of Brian White \cite{White_anno,White_stratification,White_dimension,White_nature,White_subsequent}: Since the mean curvature flow is a parabolic equation, time scales like distance squared and it is natural to equip the spacetime $\dR^{N,1}=\dR^N\times \dR$ with the metric
\begin{align}\label{d:distance}
d((x,t),(y,s))=\max (\abs{x-y},\abs{t-s}^{1/2}).
\end{align}
All neighborhoods, tubular neighborhoods, Hausdorff dimensions, etc. will be with respect to this metric, e.g. $B_r(0^{N,1})=B_r(0^N)\times(-r^2,r^2)$, $T_r(\dR^N\times\{0\})=\{(x,t)\in\dR^{N,1}: \abs{t}<r^2\}$, and $\dim(\{0\}\times\dR)=2$. We write $X=(x,t)$ for points in spacetime and $\Vol$ for the $(N+2)$-dimensional Hausdorff measure on $(\dR^{N,1},d)$.\\

For $X\in \cM$  and $r>0$ we let $\cM_{X,r}:= \cD_{1/r}(\cM-X)$, where $-X$ denotes translation and $\cD_\lambda(x,t)=(\lambda x,\lambda^2 t)$ denotes parabolic dilation. If $X_\alpha\in \cM$ converges to $X$ and $r_\alpha\to 0$, then $\cM^\alpha:=\cM_{X_\alpha,r_\alpha}$ is called a {\it blowup sequence} at $X$. After passing to a subsequence $\cM^\alpha$ converges (in the sense of Brakke flows \cite[Sect. 7]{Ilmanen_EllipticRegularization}) to a {\it limit flow} $\cN$. If $X_\alpha=X$ for all $\alpha$, then $\cN$ is called a {\it tangent flow}. Tangent flows may not be unique, but they always exist at any given point and their backwards portion $\cN^-=\{(N_t,t):t<0\}$ is always selfsimilar, i.e. $\cN^-=\cD_\lambda\cN^-$ for all $\lambda>0$. In particular, $\cN^-$ is determined by $N_{-1}$. The number $d(\cN)$ of spatial symmetries is the maximal $d$ such that $N_{-1}$ splits off a $d$-plane $V$ (possibly with multiplicity). \\

With respect to the dimensional behavior in the time direction, there are two exceptional cases. Namely, it can happen that $N_{-1}$ is a stationary cone. Then $\cN=N_{-1}\times\dR$ or $\cN=N_{-1}\times(-\infty,T]$ for some $T\geq 0$. In these cases $\cN$ is called \emph{static} or \emph{quasistatic} respectively. Otherwise $\cN$ is called a \emph{shrinker}.
Following White, we consider the number of spacetime symmetries,
\begin{equation}
D(\cN)= \left\{ \begin{array}{ll}
d(\cN) & \textrm{if $\cN$ is a shrinker or quasistatic}\\
d(\cN)+2 & \textrm{if $\cN$ is static.}\\
\end{array} \right.
\end{equation}

The number $D(\cN)$ is the dimension of the subset of points $X\in\cN$ such that $\cN_X\equiv \cN$.  We say that $\cN$ is $j$-selfsimilar, if its backwards portion is selfsimilar and $D(\cN)\geq j$. Note that every tangent flow is $0$-selfsimilar. Now, following White again, we define a stratification of the singular set $\cS(\cM)$,
$$
\cS^0(\cM)\subseteq \cS^1(\cM)\subseteq\cdots\subseteq \cS^{n+1}(\cM)\subseteq \cS(\cM)\subseteq\cM\, ,
$$
where by definition, $X\in \cS^j(\cM)$ if and only if no tangent flow at $X$ is $(j+1)$-selfsimilar. \\

For general Brakke flows of integral varifolds, recall that White proved the (parabolic) Hausdorff dimension estimate
\begin{equation}\label{White_hausdorffest}
\dim\, \cS^j(\cM)\leq j\, ,
\end{equation}
see \cite{White_stratification}. For the flow of mean-convex hypersurfaces, he proved the deep result
\begin{equation}
\cS(\cM)=\cS^{n-1}(\cM),
\end{equation}
and thus that the singular set has (parabolic) Hausdorff dimension at most $n-1$ \cite{White_dimension}. This is based on many clever arguments, in particular the ruling out of tangent flows of higher multiplicities. Some of his key steps are an expanding hole theorem, a Bernstein-type theorem and a sheeting theorem. In his later articles \cite{White_nature, White_subsequent} he gives a precise description of the singularities in this mean convex case: He proves that all tangent flows are spheres, cylinders or planes of multiplicity one.

\subsection{Quantitative regularity results for the mean curvature flow}\label{ss:quant_reg}

Given $\eta>0$ and $0<r<1$, we define a quantitative version $\cS^j_{\eta,r}(\cM)$ of the singular strata $\cS^j(\cM)$. The criterion for membership of $X\in\cM$ in $\cS^j_{\eta,r}(\cM)$ involves the behavior of $\cM$ on $B_s(X)$ for all $r\leq s\leq 1$. To state it, we fix a suitable distance function $d_{\cB}$ on the space of Brakke flows on $B_1(0^{N,1})$ (see Section \ref{prelim}).

\begin{definition}\label{d:Brakke_singularstrata}
 For each $\eta>0$ and $0<r<1$, we define the $j$-th quantitative singular stratum
\begin{align}
\cS^j_{\eta,r}(\cM):=\{X\in \cM: d_{\cB}(\cM_{X,s},\cN)>\eta\,\,{\rm for\,\, all}
\,\,r\leq s\leq 1\,
{\rm and \,\,all}\, (j+1)\text{-selfsimilar } \cN\}.
\end{align}
\end{definition}
\noindent The distance above is the Brakke flow distance introduced in Section \ref{ss:dist_Brakke}.  The quantitative singular strata satisfy
\begin{equation}
\cS^{j}_{\eta,r}(\cM)\subset \cS^{j'}_{\eta',r'}(\cM)\quad
({\rm if}\,\,j\leq j',\,\,\eta\geq\eta',\,\, r\leq r')\, ,\label{containments}
\end{equation}
and
\begin{equation}
\cS^j(\cM)=\bigcup_\eta \,\,\bigcap_r \, \cS^j_{\eta,r}(\cM)\, .\label{relation}
\end{equation}
The inclusion (\ref{containments}) follows directly from the definitions, the relation (\ref{relation}) is a bit more tricky and will be proved in Section \ref{ss:dist_Brakke}.  This is also a good test case to show that our choice of distance function $d_{\cB}$ is sensible. 

Our first main theorem gives an estimate for the volume of tubular neighborhoods of the quantitative singular strata.  In the following we let $\cB^n_\Lambda(B_2(0^{N,1}))$ denote the Brakke flows with total mass bounded by $\Lambda$ in $B_2(0^{N,1})$, see Section \ref{prelim} for a precise definition.

\begin{theorem}
\label{t:Brakke_quant_strat}
For all $\varepsilon,\eta>0$, $\Lambda<\infty$ and integers $n<N$, there exists a constant $C=C(\varepsilon,\eta,\Lambda,n,N)<\infty$ such that:
If $\cM\in\cB^n_\Lambda(B_2(0^{N,1}))$ is a Brakke flow of $n$-dimensional integral varifolds in $B_2(0^{N,1})\subset\dR^{N,1}$ with mass at most $\Lambda$, then its $j$-th quantitative singular stratum satisfies
\begin{equation}\label{Brakke_volest}
\Vol\left(T_r(\cS^j_{\eta,r}(\cM))\cap B_1(0^{N,1})\right)\leq Cr^{N+2-j-\varepsilon}\qquad (0<r<1)\, .
\end{equation}
\end{theorem}

\begin{remark} By virtue of (\ref{relation}), we recover the Hausdorff dimension estimate (\ref{White_hausdorffest}), and in fact improve this to a Minkowski dimension estimate. Of course, our theorem contains more quantitative information about the singular set than just its dimension.
\end{remark}

In our applications we focus on Brakke flows starting at $k$-convex smooth compact  embedded hypersurfaces. Building on the work of White via elliptic regularization (see Ilmanen \cite{Ilmanen_EllipticRegularization}) we prove:

\begin{theorem}\label{t:Brakke_cylinders}
Let $M_0^n\subset\mathbb{R}^{n+1}$ be a $k$-convex smooth compact embedded hypersurface and let $\cM$ be the Brakke flow starting at $M_0$. Then any $k$-selfsimilar limit flow is in fact a static multiplicity one plane. In particular, for every singular point $X\in\cS(\cM)$ all tangent flows are shrinking spheres or cylinders
\begin{equation}
\dR^j\times S^{n-j}\qquad \mathrm{with} \qquad 0\leq j < k\, .
\end{equation}
\end{theorem}

\begin{remark}
Here we tacitly assume that $\cM$ is the unique Brakke flow without mass-drop, see Section \ref{ss:meanconvex}.
\end{remark}

\noindent
The above classification of tangent flows gives a bound for the (parabolic) Hausdorff dimension of the singular set:
\begin{corollary}\label{t:Brakke_cylinderssize}
Let $\cM$ be a Brakke flow starting at a $k$-convex smooth compact embedded hypersurface. Then
\begin{equation}
\cS(\cM)=\cS^{k-1}(\cM)\qquad\textrm{and}\qquad\dim \cS(\cM)\leq k -1\, .
\end{equation}
\end{corollary}
In fact, this can be strengthened to a bound for the Minkowski dimension (see Theorem \ref{t:Brakke_regularity}).  Our primary applications of Theorem \ref{t:Brakke_cylinders} are to the $\varepsilon$-regularity of Theorem \ref{t:Brakke_epsregularity}, which itself is combined with the quantitative stratification to prove the regularity Theorem \ref{t:Brakke_regularity}.  To state this theorem we introduce the notion of the regularity scale:

\begin{definition}\label{d:regularity_scale}
For $X=(x,t)\in\cM$ we define the \emph{regularity scale} $r_{\cM}(X)$ as the supremum of $0\leq r\leq 1$ such that $M_{t'}\cap B_r(x)$ is a smooth graph for all $t-r^2< t'< t+r^2$ and such that
\begin{equation}
\sup_{X'\in\cM\cap B_r(X)} r\, \abs{A}\!(X')\leq 1\, ,
\end{equation}
where $A$ is the second fundamental form.  For $0<r<1$ we define the r-\emph{bad set}
\begin{equation}
\cB_r(\cM):=\{X=(x,t)\in\cM \, |\,  r_{\cM}(X)\leq r\}\, .
\end{equation}
\end{definition}

The relevance of this definition is that for points outside the bad set we get definite bounds on the geometry in a neighborhood of definite size. In particular, parabolic estimates give control of all the derivatives of the curvatures
\begin{equation}\label{interiorest}
\sup_{X'\in\cM\cap B_{r/2}(X)} r^{\ell+1}\, \abs{\nabla^\ell A}\!(X')\leq C_\ell\qquad \textrm{for}\quad X\in\cM\setminus \cB_r(\cM) \, .
\end{equation}

The following theorem shows that (a tubular neighborhood of) the bad set is small.
\begin{theorem}\label{t:Brakke_regularity}
Let $\cM$ be a Brakke flow starting at a $k$-convex smooth compact embedded hypersurface $M_0^n\subset \dR^{n+1}$. Then for every $\varepsilon>0$ there exists a constant $C=C(M_0,\varepsilon)<\infty$ such that we have the volume estimate
\begin{equation}
\Vol(T_r(\cB_r(\cM)))\leq C r^{n+4-k-\varepsilon}\qquad\qquad (0<r<1)\, ,
\end{equation}
for the $r$-tubular neighborhood of the bad set $B_r(\cM)$. In particular, the (parabolic) Minkowski dimension of the singular set $\cS(\cM)$ is at most $k-1$.
\end{theorem}
\begin{remark}
Note that the exponent arises as $n+4-k=\dim(\dR^{n+1,1})-(k-1)$.
\end{remark}
\begin{remark}
For the analogous results proved in \cite{CheegerNaber_Ricci} for Einstein manifolds it was possible in some situations to remove the $\varepsilon$-dependence from the righthand side.  There is no analogous situation here.
\end{remark}

As a consequence, we obtain $L^p$-estimates for the inverse regularity scale, and thus in particular $L^p$-estimates for the second fundamental form and its derivatives.
\begin{corollary}\label{t:Brakke_estimates}
Let $\cM$ be a Brakke flow starting at a $k$-convex smooth compact embedded hypersurface $M_0^n\subset \dR^{n+1}$.  Then for every $0< p<n+1-k$ there exists a constant $C=C(M_0,p)<\infty$ such that
\begin{equation}\label{Brakke_reglpestimates}
\int r_{\cM}^{-p}dM_t \leq C\qquad \mathrm{and} \qquad  \int_0^\infty\int r_{\cM}^{-(p+2)}dM_t\, dt\leq C\, .
\end{equation}
In particular, we have $L^p$-estimates for the second fundamental form,
\begin{equation}\label{Brakke_lpestimates}
\int \abs{A}^{p}dM_t \leq C\qquad \mathrm{and} \qquad  \int_0^\infty\int \abs{A}^{p+2}dM_t\, dt\leq C\, ,
\end{equation}
and also $L^p$-estimates for the derivatives of the second fundamental form,
\begin{equation}
\int \abs{\nabla^\ell A}^{\tfrac{p}{\ell+1}}dM_t \leq C_\ell\qquad \mathrm{and} \qquad  \int_0^\infty\int \abs{\nabla^\ell A}^{\tfrac{p+2}{\ell+1}}dM_t \, dt\leq C_\ell\, ,
\end{equation}
for some constants $C_\ell=C_\ell(M_0,p)<\infty$ ($\ell=1, 2, \ldots$).
\end{corollary}

\begin{remark}
Note that $n+1-k$ is the critical exponent for the shrinking cylinders $\dR^{k-1}\times S^{n+1-k}$\, .  That is, the mean curvature flow of $\dR^{k-1}\times S^{n+1-k}$ has local $L^p$ estimates on $|A|$ for all $p<n+1-k$, but not for $p>n+1-k$.  In particular, the above estimates are sharp.
\end{remark}

\subsection{Relationship with other works on the regularity of mean curvature flow}

To put things into a broader context, let us mention some other works (besides the results of White and our contribution) on the regularity and on estimates for the mean curvature flow.\\

Brakke's main regularity theorem \cite[Thm. 6.12]{Brakke_book} (see also \cite{KasaiTonegawa_regularity}), gives regularity almost everywhere, but relies on the unit density hypothesis. In codimension one, but still without any convexity assumption on the initial surface, Ilmanen \cite[Sect. 12]{Ilmanen_EllipticRegularization} (see also \cite{EvansSpruck_paper4}) proved a generic almost everywhere regularity theorem without unit density assumption. Other early fundamental contributions include the theory of the level-set flow by Evans-Spruck \cite{EvansSpruck_levelset} and Chen-Giga-Goto \cite{ChenGigaGoto_levelset}.\\

Using tools from the smooth world, the $k$-convex case (in particular for $k=2$ and $n$) has been extensively studied by Huisken-Sinestrari \cite{HuiskenSinestrari_MeanConvex,HuiskenSinestrari_Convexity,HuiskenSinestrari_Surgery}, Andrews \cite{Andrews_noncollapsing} and Sheng-Wang \cite{ShengWang_singularityprofile,Wang_convex}. This also provides alternatives for some arguments of White, at least up to the first singular time. Specializing to the $2$-convex case, there is the very interesting thesis of John Head \cite{Head_thesis}. In particular, by taking limits of the Huisken-Sinestrari surgery-solutions he obtains the same estimates (in the special case $k=2$) as we did in (\ref{Brakke_lpestimates}). Thus, there are two very long and sophisticated lines of reasoning for the $2$-convex case. One of them starts with Huisken's study of smooth solutions \cite{Huisken_convex}, moves all the way up to the surgery construction of Huisken-Sinestrari \cite{HuiskenSinestrari_Surgery} and concludes with the limiting argument of Head \cite{Head_thesis}. The other one starts with the weak Brakke solutions, goes through all the regularity theory of White and concludes with our quantitative stratification.  Both arguments give the same $L^{p}$-estimates, though we emphasize that the latter argument allows additionally for sharp $L^p$ estimates on $k$-convex flows for $k>2$, which were previously unknown, and in fact we prove the significantly stronger $L^p$ estimates on the regularity scale in (\ref{Brakke_reglpestimates}) for all $k$-convex flows.\\

Finally, there is a very interesting alternative approach by Klaus Ecker using integral estimates \cite{Ecker_firstsingtime}, where he investigates the size of the singular set at the first singular time in the $k$-convex case. This is related to our Corollaries \ref{t:Brakke_cylinderssize} and \ref{t:Brakke_estimates}. See also \cite{Ilmanen_singularities,HanSun_firstsing,XuYeZhao_firstsingu,LeSesum_firstsing} for further related results.

\section{Preliminaries}\label{prelim}

In this section we collection some standard references and results which will be needed in subsequent sections.

\subsection{Brakke flow}

Let $\cM=\{(M_t,t) : t\in I\}$ be a Brakke flow of $n$-dimensional integral varifolds in an open subset $U\subset\dR^N$ \cite{Brakke_book,Ilmanen_EllipticRegularization}. This means that $M_t$ is a one-parameter family of Radon measures on $U$ such that for almost all times $M_t$ is associated to an $n$-dimensional integer muliplicity rectifiable varifold and that
\begin{equation}\label{brakkeineq}
\overline{D}_t\int \varphi dM_t\leq \int\left(-\varphi H^2+\nabla \varphi\cdot H\right) dM_t,
\end{equation}
for all nonnegative test functions $\varphi\in C_c^1(U,\dR^+)$. Here, $\overline{D}_t$ denotes the limsup of difference quotients and $H$ denotes the mean curvature vector, which is defined via the first variation formula and exists almost everywhere at almost all times. The righthand side is interpreted as $-\infty$ whenever the integral does not exist. We assume that the mass is bounded,
\begin{equation}
M_t(U)\leq \Lambda <\infty\qquad (t\in I).
\end{equation}
This can always be achieved by passing to smaller sets and time-intervals (using that Radon-measures are finite on compact sets and a local version of the fact that the mass is decreasing under mean curvature flow). After rescaling and shifting time and space, we can assume without essential loss of generality that $U=B_2(0^N)$ and $I=(-4,4)$, i.e. that $\cM$ is defined on $B_2(0^{N,1})=B_2(0^N)\times(-4,4)$. To keep track of everything said so far we write $\cM\in\cB^n_\Lambda(B_2(0^{N,1}))$. We sometimes, when there is no danger of confusion, also denote by $M_t$ and $\cM$ the support of the measure $M_t$ and of the Brakke flow $\cM$ respectively.

\subsection{The mean-convex case}\label{ss:meanconvex}
If $M_0^n\subset\dR^{n+1}$ is a mean-convex smooth compact embedded hypersurface, then there is an essentially unique Brakke flow $\cM=\{(M_t,t)|t\geq 0\}$ starting at $M_0$.\footnote{It is unique up to sudden disappearing of connected components (mass-drop) \cite{Ilmanen_EllipticRegularization,Soner_uniqueness}. Since sudden disappearing only helps in our proofs, we exclude it from now on.} In fact, equality holds in (\ref{brakkeineq}), see Metzger-Schulze \cite{MetzgerSchulze_nomassdrop}, and the flow can also be described as the level set flow of Chen-Giga-Goto and Evans-Spruck \cite{ChenGigaGoto_levelset,EvansSpruck_levelset}.

\subsection{Localized monotonicity formula}
Let $\cM\in\cB^n_\Lambda(B_2(0^{N,1}))$ and $X_0=(x_0,t_0)\in B_1(0^{N,1})$. Let
\begin{equation}
\phi_{X_0}(x,t)=\frac{1}{(4\pi(t_0-t))^{n/2}}e^{-\abs{x-x_0}^2/4(t_0-t)}
\end{equation}
be the backwards heat kernel based at $X_0$ and consider the cutoff function
\begin{equation}
\rho_{X_0}(x,t)=\left(1-\abs{x-x_0}^2-2n(t-t_0)\right)_+^3.
\end{equation}
Then we have the following localized version of Huisken's monotonicity formula \cite{Huisken_monotonicity}:
\begin{equation}\label{Brakke_monotonicity}
\overline{D}_t\int \phi_{X_0}\rho_{X_0} dM_t\leq - \int\abs{H-\frac{(x-x_0)^\perp}{2(t-t_0)}}^2 \phi_{X_0}\rho_{X_0} dM_t\qquad (t<t_0)\, .
\end{equation}
This is proved for smooth hypersurfaces in \cite[Prop. 4.17]{Ecker_book}, but as indicated there the proof can be generalized for Brakke flows in arbitrary codimension (see also \cite[Lemma 7]{Ilmanen_singularities}). The monotonicity formula (\ref{Brakke_monotonicity}) has many very useful consequences: First, it implies bounds for the density ratios. Concretely, if $X=(x,t)\in B_1(0^{N,1})$ and $0<r\leq1/2$, then
\begin{equation}\label{density_bounds}
\frac{M_t(B_{r}(x))}{r^n}\leq 4\Lambda\, .
\end{equation}
Then, using the compactness theorem for Brakke flows \cite[Thm 7.1]{Ilmanen_EllipticRegularization}, blowup sequences $\cM^\alpha=\cM_{X,r_\alpha}$ always have a subsequential limit $\cN$, called a tangent flow, and again utilizing the monotonicity formula it follows that $\cD_{\lambda}\cN^-=\cN^-$ (see also \cite[Lemma 8]{Ilmanen_singularities}). Finally, using the notation
\begin{equation}
\Theta(\cM,X_0,\tau)=\int \phi_{X_0}\rho_{X_0} dM_{t_0-\tau}\, ,
\end{equation}
we define the Gaussian density at $X_0$ by the formula
\begin{equation}
\Theta(\cM,X_0)=\lim_{\tau\to 0}\Theta(\cM,X_0,\tau)\, .
\end{equation}
Note by monotonicity that this limit is well defined.  The Gaussian density is the important quantity in the local regularity theorem \cite{Brakke_book,White_regularitytheorem,KasaiTonegawa_regularity}.

\subsection{Distance between Brakke flows}\label{ss:dist_Brakke}
After restricting to the unit ball, the rescaled flows $\cM^\alpha=\cM_{X,r_\alpha}\, (r_\alpha\leq 1/2)$ are elements of the space $\cB^n_{4\Lambda}(B_{1}(0^{N,1}))$. Let us define a pseudometric $d_{\cB}$ as follows: Pick a countable dense subset of functions $\phi_\alpha\in C_c(B_1(0^N))$ and a countable dense subset of times $t_\beta\in(-1,1)$, and define
\begin{equation}
d_{\cB}(\cM,\cN):=\sum_{\alpha,\beta}\frac{1}{2^{\alpha+\beta}}\frac{\abs{\int \phi_\alpha\, dM_{t_\beta}-\int \phi_\alpha\, dN_{t_\beta}}}{1+\abs{\int \phi_\alpha\, dM_{t_\beta}-\int \phi_\alpha\, dN_{t_\beta}}}\, .
\end{equation}
Note that one could just as well take a countable collection of spacetime functions in the definition, though for Brakke flows it makes little difference.  The key facts about $d_{\cB}$ are that it is on the one hand weak enough to be compatible with the convergence of Brakke flows and on the other hand strong enough to distinguish between different $0$-selfsimilar solutions. To illustrate this, let us prove (\ref{relation}):
\begin{proof}[Proof of (\ref{relation})]
If $X\notin \cS^j(\cM)$, then there is a $(j+1)$-selfsimilar tangent flow, $\cM_{X,r_\alpha}\to \cN$. Since $d_{\cB}$ is weak enough this implies $d_{\cB}(\cM_{X,r_\alpha},\cN)\to 0$ and thus $X\notin \bigcup_\eta \,\,\bigcap_r \, \cS^j_{\eta,r}(\cM)$. Conversely, if $X\notin \bigcup_\eta \,\,\bigcap_r \, \cS^j_{\eta,r}(\cM)$, then we can find a sequence of $(j+1)$-selfsimilar flows $\cN^\alpha$ and scales $s_\alpha\to 0$ such that $d_{\cB}(\cM_{X,s_\alpha},\cN^\alpha)\to 0$.  Note that since we have uniform local mass bounds for $\cM_{X,s_\alpha}$ from (\ref{density_bounds}), we also have local mass bounds for $\cN^\alpha$ because of a local version of the fact that mass is decreasing in time under Brakke flow.  Therefore, after passing to subsequences we can find a $0$-selfsimilar limit $\cM$ and a $(j+1)$-selfsimilar limit $\cN$. Then $d_{\cB}(\cM,\cN)=0$, and since $d_{\cB}$ is strong enough to distinguish between $0$-selfsimilar solutions, this implies that $\cM=\cN$. Thus $\cM$ is $(j+1)$-selfsimilar. Now for $\delta_\alpha\to 0$ slowly enough, $\cM_{X,\delta_\alpha s_\alpha}$ converges to a $(j+1)$-selfimilar flow on whole $\dR^{N,1}$, and we conclude that $X\notin \cS^j(\cM)$.
\end{proof}
\section{Volume estimates for quantitative strata}\label{s:current_decomposition}

In this section, we prove Theorem \ref{t:Brakke_quant_strat}. In outline, we follow the scheme introduced in \cite{CheegerNaber_Ricci,CheegerNaber_HarmonicMinimal}. We first prove a quantitative rigidity lemma and decompose $\cM\cap B_1(0)$ into a union of sets, according to the behavior of points at different scales. By virtue of a quantitative differentiation argument, we show that the number of sets in this decomposition grows at most polynomially. We then establish a cone-splitting lemma for Brakke flows and prove, roughly speaking, that at their good scales points in  $\cS^j_{\eta,r}(\cM)$ line up along at most $j$-dimensional subspaces. Using all this, we conclude the argument by constructing a suitable covering of $\cS^j_{\eta,r}(\cM)\cap B_1(0)$ and computing its volume.

The possibility of quasistatic tangent flows causes some additional difficulties. In fact, the quasistatic case was the reason why White proved a general stratification theorem in \cite{White_stratification}, somewhat different in spirit to the previous ones in the literature. To resolve the quasistatic issue in our case, we essentially show that it suffices to cover a neighborhood of the final time-slice (see Lemma \ref{l:Brakke_quasistatic} and Section \ref{brakke_conclusion}). This also gives an alternative proof of White's result.

\subsection{Energy decomposition}\label{ss:energy_decomposition}

The goal of this subsection is to decompose $\cM\cap B_1(0)$ into a union of sets $E_{T^\beta}$, according to the behavior of points at different scales. As in \cite{CheegerNaber_Ricci,CheegerNaber_HarmonicMinimal} it will be of crucial importance that we can deal separately with each individual set $E_{T^\beta}$, all of whose points have the same $\beta$-tuple of good and bad scales.

\begin{definition}
\label{d:current_almost_conical}
A Brakke flow $\cM\in\cB^n_{\Lambda}(B_{2r}(X))$ is $(\varepsilon,r,j)$-selfsimilar at $X=(x,t)$ if there exists a $j$-selfsimilar flow $\cN$ such that
$$
d_{\cB}(\cM_{X,r},\cN)<\varepsilon \, .
$$
If $\cN$ is a shrinker with respect to a plane $V^j$, we put $W_X=(x+V)\times\{t\}$. If $\cN$ is quasistatic with respect to $V^j$ and disappears at time $T$, we put $W_X=(x+V)\times (-\infty,t+r^2T]$. If $\cN$ is static with respect to $V^{j-2}$, we put $W_X=(x+V)\times\dR$.  We say that $\cM$ is {\it $(\varepsilon,r,j)$-selfsimilar} at $X$ with respect to $W_X$.
\end{definition}

\begin{lemma}[Quantitative Rigidity]
\label{t:current_monotone_rigidity}
For all $\varepsilon >0$, $\Lambda<\infty$ and $n<N$ there exists $\delta=\delta(\varepsilon,\Lambda,n,N)>0$, such that if $\cM\in\cB^n_\Lambda(B_2(0^{N,1}))$
satisfies
\begin{align}\label{e:current_almostrigid}
\Theta(\cM,X,r^2)-\Theta(\cM,X,(\delta r)^2)\leq \delta \quad \textrm{for some} \quad X\in B_1(0^{N,1}),\, 0<r<1/2\, ,
\end{align}
then $\cM$ is $(\varepsilon,r,0)$-selfsimilar at $X$.
\end{lemma}

\begin{proof}
If not, then there exist $\cM^\alpha\in\cB^n_\Lambda(B_2(0^{N,1}))$, $X_\alpha\in B_1(0^{N,1})$ and $0<r_\alpha<1/2$ with
\begin{equation}\label{contr_assumpt}
\Theta(\cM^\alpha,X_\alpha,r_\alpha^2)-\Theta(\cM^\alpha,X_\alpha,(r_\alpha/\alpha)^2)\leq 1/\alpha\, ,
\end{equation}
but such that $\cM^\alpha:=\cM^\alpha_{X_\alpha,r_\alpha}$ satisfies $d_{\cB}(\cM^\alpha,\cN)\geq \varepsilon$ for all $0$-selfsimilar $\cN$. However, it follows from (\ref{Brakke_monotonicity}), (\ref{density_bounds}), (\ref{contr_assumpt}) and the compactness theorem for Brakke flows that, after passing to a subsequence, $\cM^\alpha\to\cN$ for some $0$-selfsimilar $\cN$. In particular, $M^\alpha_t\to N_t$ in the sense of Radon measures, for all $t\in (-1,1)$. But this implies $d_{\cB}(\cM^\alpha,\cN)< \varepsilon$ for $\alpha$ large enough, a contradiction.
\end{proof}

Now, for $X\in\cM\cap B_1(0)$ and $1/2 > r_1>r_2$, we define the \emph{$(r_1,r_2)$-Huisken energy} by
\begin{equation}
\cW_{r_1,r_2}(\cM,X):=\Theta(\cM,X,r_1^2)-\Theta(\cM,X,r_2^2)\geq 0.
\end{equation}
Given constants $0<\gamma<1/2$ and $\delta>0$ and an integer $q<\infty$ (these parameters will be fixed suitably in Section \ref{brakke_conclusion}), let $N$ be the number of $\alpha>q$ such that
$$
\cW_{\gamma^{\alpha-q},\gamma^{\alpha+q}}(\cM,X)>\delta\, .
$$
Clearly, the sum of the changes of $\Theta$ is bounded by $\Theta(\cM,X,1/4)\leq\Lambda/\pi^{n/2}$. Thus
\begin{equation}
\label{e:Nb}
N\leq (2q+1)\delta^{-1}\Lambda/\pi^{n/2}\, .
\end{equation}
Otherwise, there would be at least $\delta^{-1}\Lambda/\pi^{n/2}$ disjoint intervalls of the form $(\gamma^{\alpha-q},\gamma^{\alpha+q})$ with $\cW_{\gamma^{\alpha-q},\gamma^{\alpha+q}}(\cM,X)>\delta$. This is an instance of {\it quantitative differentiation}.

 For each point $X\in\cM\cap B_1(0)$, to keep track of its behavior at different scales, we define a $\{0,1\}$-valued sequence $(T_\alpha(X))_{\alpha\geq 1}$ as follows. By definition, $T_\alpha(X)=1$ if $\alpha\leq q$ or $\cW_{\gamma^{\alpha-q},\gamma^{\alpha+q}}(\cM,X)>\delta$, and $T_\alpha(X)=0$ if $\alpha>q$ and $\cW_{\gamma^{\alpha-q},\gamma^{\alpha+q}}(\cM,X)\leq\delta$. Then, for each $\beta$-tuple $(T_\alpha^\beta)_{1\leq \alpha\leq \beta}$, we  put
\begin{equation}
\label{Edef}
E_{T^\beta}(\cM)=\{X\in \cM\cap B_1(0)\, |\, T_\alpha(X)=T_\alpha^\beta \textrm{ for } 1\leq \alpha\leq \beta \}\, .
\end{equation}
A priori there are $2^\beta$ possible sets $E_{T^\beta}(\cM)$. However, by the above, $E_{T^\beta}(\cM)$ is empty whenever $T^\beta$ has more than
$$
Q:=(2q+1)\delta^{-1}\Lambda/\pi^{n/2}+q
$$
nonzero entries. Thus, we have constructed a decomposition of $\cM\cap B_1(0)$ into at most $\beta^Q$ (for $\beta\geq 2$) nonempty sets $E_{T^\beta}(\cM)$.

\subsection{Cone-splitting for Brakke flows}

The goal of this subsection is to prove Corollary \ref{c:Brakke_inductive} which says, roughly speaking, that at their good scales points line up in a tubular neighborhood of a well defined {\it almost planar} set. Here, the set of points that we call $\delta$-good at scales between $Ar$ and $r/A$ ($A>1$) is defined as
\begin{equation}
L_{Ar,r/A,\delta}(\cM)=\{X\in\cM\cap B_1(0): \cW_{Ar,A^{-1}r}(\cM,X)\leq\delta\} \, .
\end{equation}

A key role is played by the cone-splitting principle for Brakke flows and its quantitative version (Lemma \ref{l:Brakke_cone_splitting}). Similar ideas played a key role in \cite{CheegerNaber_Ricci,CheegerNaber_HarmonicMinimal}, and we recall here the cone-splitting principle for varifolds for the readers convenience:

\vskip2mm
\noindent
{\bf Cone-splitting principle for varifolds.}
Let $|I|$ denote a $n$-varifold on $\dR^N$ which is $j$-conical with respect to the
$j$-plane $V^j$.  Assume in addition that for some $y\not\in V^j$ that $|I|$ is also
$0$-conical with respect to $y$.  Then it follows that $|I|$ is $(j+1)$-conical
with respect to the $(j+1)$-plane ${\rm span}\{y,V^j\}$.
\vskip1mm

For Brakke flows there can be shrinkers, static and quasistatic tangent flows, and thus we have to distinguish a couple of cases:

\vskip2mm
\noindent
{\bf Cone-splitting principle for Brakke flows.}
Assume that $\cM$ is $j$-selfsimilar at $0^{N,1}$ with respect to $W$ and $0$-selfsimilar at $Y=(y,s)\notin W$. Then we have the following implications:
\begin{itemize}
\item If $W=V^j\times\{0\}$ and
\begin{itemize}
\item if $s=0$, then $y\notin V$ and $\cM$ is $(j+1)$-selfsimilar at $0^{N,1}$ with respect to ${\rm span}\{y,V^j\}\times\{0\}$.
\item if $s\neq 0$ and $y\in V$, then $\cM$ is $j$-selfsimilar at $0^{N,1}$ and quasistatic with respect to $V^j\times(-\infty, \max\{s,0\}]$.
\item if $s\neq 0$ and $y\notin V$, then $\cM$ is $(j+1)$-selfsimilar and quasistatic with respect to ${\rm span}\{y,V^j\}\times(-\infty, \max\{s,0\}]$.
\end{itemize}
\end{itemize}
\begin{itemize}
\item If $W=V^j\times(-\infty,T]$ and
\begin{itemize}
\item if $y\in V$, then $s>T$ and $\cM$ is $j$-selfsimilar at $0^{N,1}$ and quasistatic with respect to $V^j\times (-\infty,s]$.
\item  $y\notin V$, then $\cM$ is $(j+1)$-selfsimilar and quasistatic with respect to ${\rm span}\{y,V^j\}\times(-\infty, \max\{s,T\}]$.
\end{itemize}
\end{itemize}
\begin{itemize}
\item If $W=V^{j-2}\times\dR$, then $y\notin V$ and $\cM$ is $(j+1)$-selfsimilar and static with respect to ${\rm span}\{y,V^{j-2}\}\times\dR$.
\end{itemize}

From the above and an argument by contradiction, we immediately obtain the following quantitative refinement.

\begin{lemma}[Cone-splitting Lemma]
\label{l:Brakke_cone_splitting}
For all $\varepsilon,\rho>0$, $R,\Lambda<\infty$ and $n<N$ there exists a constant
$\delta=\delta(\varepsilon,\rho,R,\Lambda,n,N)>0$ with the following property. If $\cM \in \cB^n_\Lambda(B_{2R}(0^{N,1}))$ satisfies
\begin{enumerate}
 \item $\cM$ is $(\delta,R,j)$-selfsimilar at $0^{N,1}$ with respect to $W$.
 \item There exists $Y=(y,s)\in B_1(0^{N,1})\setminus T_{\rho}(W)$ such that $\cM$ is
$(\delta,2,0)$-selfsimilar at $Y$,
\end{enumerate}
then we have the following implications:
\begin{itemize}
\item If $W=V^j\times\{0\}$ and
\begin{itemize}
\item if $\abs{s}<\rho^2$, then $d(y,V)\geq\rho$ and $\cM$ is $(\varepsilon,1,j+1)$-selfsimilar at $0^{N,1}$ with respect to ${\rm span}\{y,V^j\}\times\{0\}$.
\item if $\abs{s}\geq\rho^2$ and $d(y,V)<\rho$, then $\cM$ is $(\varepsilon,1,j)$-selfsimilar at $0^{N,1}$ with respect to $V^j\times(-\infty, \max\{s,0\}]$.
\item if $\abs{s}\geq\rho^2$ and $d(y,V)\geq\rho$, then $\cM$ is $(\varepsilon,1,j+1)$-selfsimilar at $0^{N,1}$ with respect to ${\rm span}\{y,V^j\}\times(-\infty, \max\{s,0\}]$.
\end{itemize}
\end{itemize}
\begin{itemize}
\item If $W=V^j\times(-\infty,T]$ and
\begin{itemize}
\item if $d(y,V)<\rho$, then $s\geq T+\rho^2$ and $\cM$ is $(\varepsilon,1,j)$-selfsimilar at $0^{N,1}$ with respect to $V^j\times (-\infty,s]$.
\item if $d(y,V)\geq\rho$, then $\cM$ is $(\varepsilon,1,j+1)$-selfsimilar at $0^{N,1}$ with respect to ${\rm span}\{y,V^j\}\times(-\infty, \max\{s,T\}]$.
\end{itemize}
\end{itemize}
\begin{itemize}
\item If $W=V^{j-2}\times\dR$, then $d(y,V)\geq \rho$ and $\cM$ is $(\varepsilon,1,j+1)$-selfsimilar with respect to ${\rm span}\{y,V^{j-2}\}\times\dR$.
\end{itemize}
\end{lemma}

Using also Lemma \ref{t:current_monotone_rigidity}, by induction/contradiction we now obtain:

\begin{corollary}[Line-up in tubular neighborhoods]
\label{c:Brakke_inductive}
For all $\mu,\nu>0$, $\Lambda<\infty$ and $n<N$ there exist
$\delta=\delta(\mu,\nu,\Lambda,n,N)>0$ and $A=A(\mu,\nu,\Lambda,n,N)<\infty$ such that the following holds:  If $\cM \in \cB^n_\Lambda(B_2(0^{N,1}))$ and $X\in L_{Ar,r/A,\delta}(\cM)$ for some $r\leq 1/A$, then there exists $0\leq \ell\leq n+2$ and $W_X^\ell$ such that
\begin{enumerate}
\item $\cM$ is $(\mu,r,\ell)$-selfsimilar at $X$ with respect to $W^\ell_{X}$\, ,
\item $L_{Ar,r/A,\delta}(\cM)\cap B_{r}(X)\subseteq T_{\nu r}(W^\ell_{X})$\, .
\end{enumerate}
\end{corollary}

To deal with the quasistatic case we need another lemma. It is a quantitative version of the fact that points in a quasistatic flow at times earlier than the vanishing time look like static points when viewed at small enough scales.

\begin{lemma}[Quantitative behavior in the quasistatic case]
\label{l:Brakke_quasistatic}
For all $\varepsilon,\gamma>0$, $\Lambda<\infty$ and $n<N$ there exists $\delta=\delta(\varepsilon,\gamma,\Lambda,n,N)>0$, such that the following holds: If $\cM\in\cB^n_\Lambda(B_2(0^{N,1}))$ is $(\delta,1,\ell)$-selfsimilar at $0^{N,1}$ with respect to $W=V^\ell\times(-\infty,T]$ and if $Y=(y,s)\in\cM\cap B_{1-2\gamma}(0^{N,1})$ with $s\leq T-(2\gamma)^2$ then $\cM$ is $(\varepsilon,\gamma,\ell+2)$-selfsimilar at $Y$ with respect to $W=(y+V^\ell)\times\dR$.
\end{lemma}

\begin{proof}
If not, passing to limits we obtain a flow $\cM$ that is $\ell$-selfsimilar on $B_1(0^{N,1})$ with respect to $W=V^\ell\times (-\infty,T]$ and a point $Y=(y,s)\in \overline{B}_{1-2\gamma}(0^{N,1})$ with $s\leq T-(2\gamma)^2$ such that $\cM$ is not $(\varepsilon,\gamma,\ell+2)$-selfsimilar at $Y$ with respect to $W=(y+V^\ell)\times\dR$, a contradiction.
\end{proof}

\subsection{Conclusion of the argument}\label{brakke_conclusion}
\begin{proof}[Proof of Theorem \ref{t:Brakke_quant_strat}]
Let $\varepsilon,\eta,\Lambda,n,N,\cM$ be as in the statement of the theorem. It is convenient to choose $\gamma:=c_0(N)^{-\frac{2}{\varepsilon}}$, where $c_0(N)$ is a geometric constant that only depends on the dimension and will appear below (roughly a doubling constant). Now we apply Corollary \ref{c:Brakke_inductive} with $\nu=\gamma/2$ and $\mu\leq\eta$ small enough such that also the below application of Lemma \ref{l:Brakke_quasistatic} is justified, and get constants $\delta$ and $A$. Choose an integer $q$, such that $\gamma^q\leq 1/(2A)$. Setting $Q:=\lfloor (2q+1)\delta^{-1}\Lambda/\pi^{n/2}\rfloor+q$, from the argument in Section \ref{ss:energy_decomposition} we get a decomposition of $\cM\cap B_1(0)$ into at most $2\beta^Q$ nonempty sets $E_{T^\beta}(\cM)$.  The factor 2 is just to make it also work for $\beta=1$.

\begin{lemma}[Covering Lemma]
\label{l:Brakke_covering}
There exists $c_0(N)<\infty$ such that each set $\cS^j_{\eta,\gamma^{\beta}}(\cM)\cap E_{T^\beta}(\cM)$ can be covered by at most $c_0(c_0\gamma^{-(N+2)})^Q(c_0\gamma^{-j})^{\beta-Q}$ balls of radius $\gamma^\beta$.
\end{lemma}

\begin{proof}
We will recursively define a covering. For $\beta=0$ pick some minimal covering of $\cS^j_{\eta,\gamma^{0}}(\cM)\cap B_1(0)$ by balls of radius $1$ with centers in  $\cS^j_{\eta,\gamma^{0}}(\cM)\cap B_1(0)$. Note that $\cS^j_{\eta,\gamma^{\beta+1}}(\cM)\subset \cS^j_{\eta,\gamma^{\beta}}(\cM)$. Let $T^{\beta}$ be the $\beta$-tupel obtained from dropping the last entry from $T^{\beta+1}$. Then we also have $E_{T^{\beta+1}}(\cM)\subset E_{T^{\beta}}(\cM)$.

\vskip2mm
\noindent
{\bf Recursion step.}
For each ball $B_{\gamma^{\beta}}(X)$ in the covering of $\cS^j_{\eta,\gamma^{\beta}}(\cM)\cap E_{T^{\beta}}(\cM)$,
take a minimal covering of $B_{\gamma^{\beta}}(X)\cap \cS^j_{\eta,\gamma^{\beta+1}}(\cM)\cap  E_{T^{\beta+1}}(\cM)$
by balls of radius $\gamma^{\beta+1}$ with centers in
$B_{\gamma^{\beta}}(X)\cap \cS^j_{\eta,\gamma^{\beta+1}}(\cM)\cap  E_{T^{\beta+1}}(\cM)$.
\vskip2mm

Let us now explain that this covering has indeed the desired properties. First observe that, for all $\beta$, the number of balls in a minimal covering from the recursion step is at most
\begin{equation}
c(N)\gamma^{-(N+2)}\, .
\end{equation}
However, if $T_\beta(X)=0$, then $X\in L_{2A\gamma^{\beta},2\gamma^\beta/A,\delta}(\cM)$ and Corollary \ref{c:Brakke_inductive} gives us $0\leq \ell\leq n+2$ and $W_X^\ell$ such that
\begin{enumerate}
\item $\cM$ is $(\mu,2\gamma^\beta,\ell)$-selfsimilar at $X$ with respect to $W^\ell_{X}$\, ,
\item $L_{2A\gamma^\beta,2\gamma^\beta/A,\delta}(\cM)\cap B_{2\gamma^\beta}(X)\subseteq T_{\gamma^{\beta+1}}(W^\ell_{X})$ .
\end{enumerate}
Since $X\in\cS^j_{\eta,\gamma^{\beta}}(\cM)$ we must have $\ell\leq j$. Since $E_{T^\beta}(\cM)\subset L_{2A\gamma^\beta,2\gamma^\beta/A,\delta}(\cM)$ this implies the following better estimate for the number of balls in a minimal covering:
\begin{equation}\label{Brakke:bettervolume}
c(N)\gamma^{-j}\, .
\end{equation}
Indeed, the estimate is clear in the cases $W_X^\ell=(x+V^\ell)\times\{t\}$ and $W_X^\ell=(x+V^{\ell-2})\times\dR$, the case $W_X^\ell=(x+V^{\ell})\times(-\infty,T]$ requires some extra thought, but in fact only if $T\leq t+(2\gamma^\beta)^2$ and $\ell\geq j-1$ which we will assume now. So, if $Y=(y,s)\in B_{\gamma^{\beta}}(X)\cap \cS^j_{\eta,\gamma^{\beta+1}}(\cM)\cap  E_{T^{\beta+1}}(\cM)$, then by Lemma \ref{l:Brakke_quasistatic} we conclude $s\geq T-(2\gamma^{\beta+1})^2$ and thus (\ref{Brakke:bettervolume}) holds also in the quasistatic case.
By the quantitative differentiation argument, the better estimate (\ref{Brakke:bettervolume}) applies with at most $Q$ exceptions. This proves the lemma.
\end{proof}

We will now conclude the proof of Theorem \ref{t:Brakke_quant_strat} by estimating the volume of the covering. The volume of balls in $\dR^{N,1}$ satisfies
\begin{equation}
\label{e:harm_volume}
\Vol(B_{\gamma^\beta}(X))= w_N(\gamma^{\beta})^{N+2}\, ,
\end{equation}
which together with the choice of $\gamma$ and the fact that polynomials grow slower than exponentials, i.e. with
$$
c_0^\beta=(\gamma^\beta)^{-\frac{\varepsilon}{2}}\, ,
$$
$$
\beta^Q\leq c(N,Q)(\gamma^\beta)^{-\frac{\varepsilon}{2}}\, ,
$$
gives (recalling again the decomposition of $\cM\cap B_1(0)$ into at most $2\beta^Q$ nonempty sets $\cS^j_{\eta,\gamma^{\beta}}(\cM)\cap E_{T^\beta}(\cM)$ and the Covering Lemma \ref{l:Brakke_covering})
\begin{equation}
\label{e:main}
\begin{aligned}
\Vol(\cS^j_{\eta,\gamma^\beta}(\cM)\cap B_1(0))
&\leq 2\beta^Q \left[c_0  (c_0\gamma^{-(N+2)})^{Q}  (c_0\gamma^{-j})^{\beta-Q} \right] w_N(\gamma^\beta)^{N+2}\\
&\leq c(N,Q,\varepsilon)  \beta^Q c_0^{\beta}(\gamma^\beta)^{N+2-j}\\
&\leq c(N,Q,\varepsilon)(\gamma^\beta)^{N+2-j-\varepsilon}\, .
\end{aligned}
\end{equation}
From the above, for all $0<r<1$, we get
$$
 \begin{aligned}
 \Vol(\cS^k_{\eta,r}(\cM)\cap B_1(0))
&\leq  c(n,Q,\varepsilon) (\gamma^{-1}r)^{N+2-j-\varepsilon}\\
& \leq c(\varepsilon,\eta,\Lambda,n,N) r^{N+2-j-\varepsilon}\, .
\end{aligned}
$$
It follows that
$$
\Vol\left(T_r(\cS^j_{\eta,r}(\cM))\cap B_1(0)\right)\leq Cr^{N+2-j-\varepsilon}\qquad (0<r<1)\, .
$$
for another constant $C=C(\varepsilon,\eta,\Lambda,n,N)$, and this finishes the proof of Theorem \ref{t:Brakke_quant_strat}.
\end{proof}

\section{Elliptic regularization and $k$-convexity}

In this section we prove Theorem \ref{t:Brakke_cylinders}. We consider a Brakke flow (level-set flow) $\cM=\{(M_t,t)\, | \, t\geq 0\}$ starting at a $k$-convex smooth compact embedded hypersurface $M_0^n\subset\dR^{n+1}$. The flow is smooth up to the first singular time $T>0$. Recall that $k$-convexity means that the sum of the $k$ smallest principal curvatures is nonnegative. In particular, the mean curvature is nonnegative. It follows from the evolution equations that the second fundamental form behaves as
\begin{equation}
\partial_t A^i_j=\Delta A^i_j + \abs{A}^2\! A^i_j \qquad \qquad \partial_t h = \Delta h + \abs{A}^2\! h \, ,
\end{equation}
and hence that $k$-convexity is preserved up to the first singular time. In fact,
\begin{equation}\label{strictellconvex}
\lambda_1+\ldots +\lambda_{k}\geq c h\qquad \qquad h\geq c \, ,
\end{equation}
for some $c>0$ and let's say $T/4\leq t< T$ (see e.g. \cite[Prop 2.6]{HuiskenSinestrari_Surgery} for more detailed explanations). It is immediate from (\ref{strictellconvex}) that shrinking cylinders $\dR^j\times S^{n-j}$ with $j\geq k$ cannot arise as tangent flows (or special limit flows) at the first singular time. To rule them out also at subsequent times we will use elliptic regularization (see Ilmanen \cite{Ilmanen_EllipticRegularization}) following White \cite{White_subsequent} closely.\\

\begin{proof}[Proof of Theorem \ref{t:Brakke_cylinders}]

The idea, motivated by the work of \cite{White_subsequent}, is that because the flows are $k$-convex, we have that $M_t\times \dR$ arises as limit for $\alpha\to\infty$ of a family of \emph{smooth} flows $N_t^\alpha$, to which the maximum principle can be applied. 
 In more detail, let $\Omega\subset R^{n+1}$ be the region bounded by $M_{T/4}$. Let $K$ be a slightly smaller compact set, say the closure of the region bounded by $M_{T/2}$. Let $N^\alpha_0$ be an integral current in $\overline{\Omega}\times \dR$ that minimizes the functional
\begin{equation}
N\mapsto\int_N e^{-\alpha x\cdot e_{n+2}}\, dA
\end{equation}
subject to $\partial N=\partial \Omega \times \{0\}$. Since $\partial \Omega$ is mean-convex, $N^\alpha_0$ is given by the graph of a smooth function $f_\alpha:\overline{\Omega}\to[0,\infty)$ (see the appendix of \cite{White_subsequent} for a proof). Furthermore, $N^\alpha_0$ satisfies the Euler-Lagrange equation,
\begin{equation}\label{eleqn}
H=-\alpha e_{n+2}^\perp\, ,
\end{equation}
and the mean curvature is strictly positive. Using the Euler-Lagrange equation (\ref{eleqn}) it follows that $h^{-1}A^i_j$ satisfies
\begin{equation}
\Delta\frac{A}{h}=-2\langle\nabla\log h, \nabla\frac{A}{h}\rangle \, .
\end{equation}
Note that $(\lambda_1+\ldots+\lambda_{k+1})/h$ is a concave function on the space of symmetric matrices with positive trace. Thus, the tensor maximum principle (see e.g. \cite{Hamilton_maximumprinciple}) implies that the minimum of the function $(\lambda_1+\ldots+\lambda_{k+1})/h$ over $N^\alpha_0\cap(K\times\dR)$ is attained on $N^\alpha_0\cap(\partial K\times\dR)$.

Note that $N^{\alpha}_t:=\textrm{graph}(f_\alpha-\alpha t)$ is a family of surfaces (with boundary) in $\Omega\times\dR$ translating by mean curvature. For $\alpha\to\infty$ the flows $N_t^\alpha$ converge to the Brakke flow $M_t\times \dR$. By the local regularity theorem \cite{Brakke_book,White_regularitytheorem} the convergence is smooth at regular points. Splitting off the $\dR$-factor amounts to replacing $k+1$ by $k$ and we conclude that
\begin{equation}\label{ellregest}
\inf_{X=(x,t)\in\cM\setminus\cS(\cM)\,\mathrm{with}\, t\geq T/2} \frac{\lambda_1+\ldots+\lambda_{k}}{h}(X)\geq c\, .
\end{equation}
Inequality (\ref{ellregest}) implies that shrinking cylinders $\dR^j\times S^{n-j}$ with $j\geq k$ cannot arise as tangent flows of $\cM$. Recalling that White already proved that the only possible tangent flows in the mean-convex case are shrinking cylinders, spheres and planes, this completes the proof of the classification of the tangent flows in the $k$-convex case.\\

It remains to deal with more general limit flows $\cM_{X_\alpha,r_\alpha}\to\cN$ where the point $X_\alpha$ is allowed to vary. Note that the class of limit flows is strictly larger, and includes for example also the rotationally symmetric translating solitons. However, in our case we have the extra assumption that $\cN$ is $k$-selfsimilar. In particular, its backwards portion is selfsimilar. Thus, by \cite[Thm. 1]{White_nature} $\cN$ is either a static multiplicity one plane or a shrinking sphere or cylinder. In the case $n\geq 7$ this is stated there only for so-called special limit flows, but using White's work \cite{White_subsequent} it also holds for general limit flows.  Now, since $\cN$ is $k$-symmetric, it must either be a multiplicity one plane or a shrinking cylinder $\dR^j\times S^{n-j}$ with $j\geq k$. As before, these cylinders are excluded by (\ref{ellregest}). This finishes the proof of Theorem \ref{t:Brakke_cylinders}.
\end{proof}

\section{Estimates for $k$-convex mean curvature flows}

It is now easy to finish the proofs of Theorem \ref{t:Brakke_regularity} and Corollary \ref{t:Brakke_estimates}. For this, we need the following new $\varepsilon$-regulartiy theorem. Roughly speaking, it says that enough approximate degrees of symmetry imply regularity.

\begin{theorem}[$\varepsilon$-regularity]\label{t:Brakke_epsregularity}
Let $\cM$ be a Brakke flow starting at a $k$-convex smooth compact embedded hypersurface $M_0^n\subset \dR^{n+1}$. Then there exists an $\varepsilon=\varepsilon(M_0)>0$ such that: If $\cM$ is $(\varepsilon,\varepsilon^{-1}r,k)$-selfsimilar at some $X\in\cM$ for some $0<r\leq \varepsilon$, then $r_{\cM}(X)\geq r$.
\end{theorem}

\begin{proof}
If not, there are $X_\alpha\in\cM$ and $r_\alpha\leq 1/\alpha$ such that $\cM$ is $(1/\alpha,\alpha r_\alpha,k)$-selfsimilar at $X_\alpha$ but $r_{\cM}(X_\alpha)<r_\alpha$. Consider the blowup sequence $\cM^\alpha:=\cM_{X_\alpha,r_\alpha}$. On the one hand, it satisfies $r_{\cM^\alpha}(0)<1$. On the other hand, after passing to a subsequence, $\cM^\alpha$ converges to some $k$-selfsimilar limit flow $\cN$. By Theorem \ref{t:Brakke_cylinders}, $\cN$ must be a multiplicity one plane. By the local regularity theorem \cite{Brakke_book,White_regularitytheorem} the convergence is smooth. Thus, $r_{\cM^\alpha}(0)\geq 1$ for $\alpha$ large enough, a contradiction.
\end{proof}

\begin{proof}[Proof of Theorem \ref{t:Brakke_regularity}]
Since the initial surface is smooth, the regularity scale is bounded below for small times. By comparison, $\cM$ remains in a compact spatial region and vanishes in finite time. Thus, we have reduced the problem to the local setting of Theorem \ref{t:Brakke_quant_strat}. After this reduction, the $\varepsilon$-regularity theorem implies $\cB_r(\cM)\subset \cS^{k-1}_{\eta,\eta^{-1}r}(\cM)$ for $\eta$ small enough and the claim follows from the volume estimate of Theorem \ref{t:Brakke_quant_strat} and the $\varepsilon$-regularity Theorem \ref{t:Brakke_epsregularity}.
\end{proof}

\begin{proof}[Proof of Corollary \ref{t:Brakke_estimates}]
Using the layer-cake formula, this follows immediately from Theorem \ref{t:Brakke_regularity}, the density bounds (\ref{density_bounds}) and the interior estimates (\ref{interiorest}).
\end{proof}

{\small

\bibliographystyle{plain}

}

\end{document}